\documentclass[10pt,a4paper]{amsart}
\usepackage[margin=3.3cm]{geometry}
\usepackage{lineno,hyperref}
\usepackage{amsmath}
\usepackage{amssymb}
\usepackage[numbers]{natbib}
\usepackage[mathscr]{euscript}
\usepackage{enumerate}
\usepackage{xspace}
\usepackage{color}
\usepackage{enumerate}
\usepackage{enumitem}

\newcommand{\q}{Q}
\DeclareMathAlphabet{\mathpzc}{OT1}{pzc}{m}{it}

\theoremstyle{plain}

\newtheorem{SA}{Standing Assumption}
\newtheorem{theorem}{Theorem}[section]
\newtheorem{lemma}[theorem]{Lemma}
\newtheorem{proposition}[theorem]{Proposition}
\newtheorem{corollary}[theorem]{Corollary}

\newtheorem{definition}[theorem]{Definition}

\newtheorem{remark}[theorem]{Remark}
\newtheorem{example}[theorem]{Example}

\renewenvironment{proof}{{\parindent 0pt \it{ Proof:}}}{\mbox{}\hfill\mbox{$\Box\hspace{-0.5mm}$}\vskip 16pt}

\newcommand{\of}{[\hspace{-0.065cm}[}
\newcommand{\gs}{]\hspace{-0.065cm}]}

\numberwithin{equation}{section}

\newcommand{\dd}{\operatorname{d}\hspace{-0.05cm}}

\newcommand{\cadlag}{c\`adl\`ag }
\newcommand{\p}{P}
\newcommand{\F}{\mathbf{F}}
\newcommand{\1}{\mathbf{1}}

\newcommand{\G}{\mathbf{G}}
\newcommand{\M}{\mathscr{M}^{\textup{sp}}}
\newcommand{\K}{\mathbf{K}}

\newcommand{\E}{E}

\begin{document}

\title[SPEMMs for \(\mathscr{H}\)-SII Models]{Structure Preserving Equivalent Martingale \\Measures for \(\mathscr{H}\)-SII Models}

\author[D. Criens]{David Criens}
\address{D. Criens - Technical University of Munich, Department of Mathematics, Germany}
\email{david.criens@tum.de}
\thanks{D. Criens - Technical University of Munich, Department of Mathematics, Germany,  \texttt{david.criens@tum.de}.}


\date{\today}
\maketitle

\frenchspacing
\pagestyle{myheadings}

	\begin{abstract}
		In this article we relate the set of structure preserving equivalent martingale measures \(\M\) for financial models driven by semimartingales with conditionally independent increments to a set of measurable and integrable functions \(\mathscr{Y}\). 
		More precisely, we prove that \(\M \not = \emptyset\) if, and only if, \(\mathscr{Y}\not = \emptyset\),
		and connect the sets \(\M\) and \(\mathscr{Y}\) to the semimartingale characteristics of the driving process.
		As examples we consider integrated L\'evy models with independent stochastic factors and time-changed L\'evy models and derive mild conditions for \(\M \not = \emptyset\).
	\end{abstract}

	\section{Introduction}\label{The Financial Model}
	A class of stochastic models which reflects many statistical observations and yet has good analytical properties is the class of so-called \(\mathscr{H}\)-SII models. The stock price process \(S\) is defined by 
	\(
	S = \exp(X),
	\)
	where \(X\) is a semimartingale with \(\mathscr{H}\)-conditionally independent increments (\(\mathscr{H}\)-SII). 
	Examples of \(\mathscr{H}\)-SII models are exponential L\'evy models and the stochastic volatility models suggested by \cite{RSSB:RSSB282,CGMY03,Heston01041993,stein1991stock}.
	
	We highlight that for pure-jump exponential L\'evy models Eberlein and Jacod \cite{EJ97} established a precise description of the set \(\M\) of SPEMMs in terms of a set of deterministic functions. 
	This result is 
	mathematically sharp and engages through its simple deterministic nature. 
	
	We show that such a result also holds for \(\mathscr{H}\)-SII models.
	More precisely, we prove that there exists a set of measurable and integrable functions \(\mathscr{Y}\) such that for each element in \(\mathscr{Y}\) there exists a corresponding measure in \(\M\) and vise versa. 
	
	To the best of our current knowledge, for \(\mathscr{H}\)-SII models the set \(\M\) was only studied for individual models, cf., e.g.,  \cite{doi:10.1080/14697680600573099,Kassberger2010,MAFI:MAFI175,doi:10.1137/S0040585X97981032}, and not from a general perspective.
	We stress that some key techniques of previous approaches to do not apply to a general setting.
	For example, in the discussion of \(\M\) for the Barndorf-Nielson and Shephard model in \cite{MAFI:MAFI175}, the following fact is used: If \(\xi\) is a process independent of a Brownian motion \(W\), then conditioned on \(\xi\) the random variable \(\int_0^T \xi_s \dd W_s\) is Gaussian distributed. 
	This claim relies on the fact that \(W\) stays a Brownian motion under the enlarged filtration which includes all informations on \(\xi\), cf. Appendix~\ref{scope}. 
	Using that \(\int_0^T\xi_s \dd W_s\) is Gaussian, the martingale property of a candidate density process for an element of \(\M\) can be computed directly. 
	In more general situations one cannot hope to perform that kind of computations. Hence, a more robust argumentation is necessary.

	At the core of the proof of Eberlein and Jacod \cite{EJ97} 
	is the fact that an exponential L\'evy process is a martingale if, and only if, it is a local martingale.
	This observation is also true in the case of \(\mathscr{H}\)-SIIs with absolutely continuous characteristics, cf. \cite{KMK10}. 
	By reducing the claim to semimartingales with independent increments (SIIs), for which the result was proven by Kallsen and Muhle-Karbe \cite{KMK} 
	exploiting a technique based on a change of measure, we generalize this observation to general \(\mathscr{H}\)-SIIs.
	In order to use this fact to construct a density processes of a measure in \(\M\), one has to show that the logarithm of a candidate density process is an \(\mathscr{H}\)-SII. 
	This, however, requires in depth measurability considerations, cf.~Appendix~\ref{MLemmata A}.
	On the other hand, to obtain necessary conditions for \(\M \not = \emptyset\), we benefit from Girsanov's theorem and deep results on local absolute continuity of laws of semimartingales as given in \cite{KLS-LACOM1}.

	Let us shortly summarize the structure of the article.
	In Section \ref{Semimartingales with conditionally independent increments} we introduce our mathematical setting. 
	Our main result is given in Section \ref{SPEMM}. We discuss the simplified situation of a quasi-left continuous driving process with continuous local martingale part in Section \ref{The Case of Local Characteristics}.
	In Section \ref{CS} we present examples such as a Black-Scholes-type model with independent stochastic volatility and an exponential L\'evy model with independent stochastic time-change. 
	The proof of our main result is given in Section~\ref{Proofs}.

	\section{Structure Preserving Equivalent Martingale Measures}\label{SCII}
	Let \(T > 0\) be a finite time horizon.
	All processes in this article are indexed on \([0, T]\).
	We fix a not-necessarily right-continuous filtration \((\mathscr{F}^{o}_t)_{t \in [0, T]}\) on a measurable space \((\Omega, \mathscr{F})\) and set \(\mathscr{F}_t \triangleq \mathscr{F}^{o}_{t+}.\)
	Throughout the entire article let \((\Omega, \mathscr{F}, \F \triangleq (\mathscr{F}_t)_{t \in [0, T]}, \p)\) be the underlying filtered probability space.
	Note that we do not assume the \emph{usual conditions}.
	For a careful discussion of the general theory of stochastic processes without assuming the usual conditions we refer to the monographs \cite{HWY,JS}. 
	
	Let \(\mathscr{H}\subseteq\mathscr{F}\) and consider the enlarged filtration \(\G \triangleq (\mathscr{G}_t)_{t \in [0, T]}\) given by
	\(
	\mathscr{G}_t \triangleq \mathscr{G}^o_{t+}\), where
	\(\mathscr{G}^o_t \triangleq \mathscr{F}^{o}_t \vee \mathscr{H}.\)
	We impose the following assumption on the underlying filtered space.
	\begin{SA}\label{SA1}
		The space \(\Omega\) is Polish and \(\mathscr{F}\) is its topological Borel \(\sigma\)-field. Moreover, for all \(t \in [0, T]\) the \(\sigma\)-fields \(\mathscr{H}\) and \(\mathscr{F}^{o}_t\) are countably generated.
	\end{SA}
	The following lemma shows that many \(\sigma\)-fields are countably generated.
	\begin{lemma}\label{lemma countable generated}
		Let \((Y_t)_{t \geq 0}\) be a right- or left-continuous process with values in a Polish space. Then for \(t \in [0, T]\) the \(\sigma\)-field \(\sigma(Y_s, s \in [0, t])\) is countably generated.
	\end{lemma}
	\begin{proof}
		It suffices to note that \(\sigma(Y_s, s \in [0, t]) = \sigma(Y_{s \wedge t}, s \in\mathbb{Q}_+)\).
	\end{proof}
	
	An important consequence of Standing Assumption \ref{SA1} is the existence of a regular conditional probability \(P(\cdot|\mathscr{H})(\cdot)\) from \((\Omega, \mathscr{H})\) to \((\Omega, \mathscr{F})\), cf., e.g., \cite[Theorem 9.2.1]{stroock2010}. 
	More precisely, \(P(\cdot|\mathscr{H})(\cdot)\) satisfies the following:
	\begin{enumerate}
		\item[\textup{(i)}]
		For all \(\omega \in \Omega\), \(A \mapsto P(A|\mathscr{H})(\omega)\) is a probability measure on \((\Omega, \mathscr{F})\).
		\item[\textup{(ii)}]
		For all \(A \in \mathscr{F}\), \(\omega \mapsto P(A|\mathscr{H})(\omega)\) is \(\mathscr{H}\)-measurable.
		\item[\textup{(iii)}]
		For all \(A \in \mathscr{F}\) the random variable \(P(A|\mathscr{H})\) is a \(P\)-version of
		\(
		E[\1_A|\mathscr{H}].
		\)
		\item[\textup{(iv)}]
		There exists a \(\p\)-null set \(N\in \mathscr{H}\) such that for all \(\omega \in \complement N\) and all \(G \in \mathscr{H}\) we have
		\begin{align}\label{varad}
		P(G|\mathscr{H})(\omega) = \1_G(\omega).
		\end{align}
	\end{enumerate}
	Part (iv) uses the assumption that the \(\sigma\)-field \(\mathscr{H}\) is countably generated.
	Let us shortly note two elementary observations.
	\begin{remark}\label{Remark SA}
		\begin{enumerate}
			\item[\textup{(i)}]
			For all \(\mathscr{F}\)-measurable functions \(Y \colon \Omega \to \mathbb{R}^+\) the random variable \(\int Y (\omega) P(\dd \omega|\mathscr{H})\)
			is a \(\p\)-version of the conditional expectation \(\E[Y|\mathscr{H}]\).
			\item[\textup{(ii)}]
			For all \(P\)-a.s. events \(A \in \mathscr{F}\) there exists a \(\p\)-null set \(N_A \in \mathscr{H}\) such that for all \(\omega \in \complement N_A\) we have \(P(A|\mathscr{H})(\omega) =1\).
		\end{enumerate}
	\end{remark}
	\subsection{Semimartingales with \(\mathscr{H}\)-Conditionally Independent Increments}\label{Semimartingales with conditionally independent increments}
	As observed by Grigelionis \cite{Grigelionis1975} semimartingales with \(\mathscr{H}\)-conditionally independent increments (\(\mathscr{H}\)-SIIs) can be characterized by measurability properties of their characteristics. 
	Before we give a precise statement let us clarify some terminology. 
	We say that \(B \in \mathscr{V}\) has an \(\mathscr{H}\)-measurable version, if for each \(t \in [0, T]\) the random variable \(B_t\) has an \(\mathscr{H}\)-measurable version.
	Denote by \(\mathscr{I}\) the set of all Borel functions \(g \colon \mathbb{R} \to\mathbb{R}\) with \(|g(x)| \leq 1 \wedge |x|^2\).
	We say that a compensator \(\nu\) of a random measure of jumps has an \(\mathscr{H}\)-measurable version, if for all \(t \in [0, T]\) and all \(g \in \mathscr{I}\) the random variable \(\nu([0, t] \times g)\) has an \(\mathscr{H}\)-measurable version.
	
	In this article we will fix a truncation function \(h \colon \mathbb{R} \to \mathbb{R}\). Whenever we talk about (semimartingale) characteristics, we refer to the characteristics corresponding to \(h\).
	
	\begin{definition}
		We call a real-valued \((\G,\p)\)-semimartingale which starts at zero an \((\mathscr{H}, \F, \p)\)-SII if its \((\G, \p)\)-characteristics have an \(\mathscr{H}\)-measur\-able \(\p\)-version.
	\end{definition}

	The following lemma can be used to deduce claims concerning \(\mathscr{H}\)-SIIs from results concerning semimartingales with independent increments.
	It is a consequence of \cite[Lemma II.6.13, Corollary II.6.15]{JS}. 
	\begin{lemma}\label{lemma kernel}
		
		A process \(Y\) is an \((\mathscr{H}, \F, \p)\)-SII if and only if there exists a \(\p\)-null set \(N \in \mathscr{F}\) such that for all \(\omega \in \complement N\) the process \(Y\)
		is a \((\{\Omega, \emptyset\}, \G, \p(\cdot|\mathscr{H})(\omega))\)-SII.
		In this case, the \((\G, \p(\cdot|\mathscr{H})(\omega))\)-characteristics of \(Y\) coincide with the \((\G, \p)\)-characteristics.
	\end{lemma}
	
	\subsection{Structure Preserving Equivalent Martingale Measures}\label{SPEMM}
	Let us now describe the class of financial models considered in this article. 
	\begin{SA}\label{SA}
		The process \(X\) is an \((\mathscr{H}, \F, \p)\)-SII and also an \((\F, \p)\)-semimartingale whose \((\F, \p)\)- and \((\G, \p)\)-characteristics coincide. 
	\end{SA}
	We discuss Standing Assumption \ref{SA} in Appendix \ref{scope} and give examples.
	The characteristics of \(X\) are denoted by \((B^X, C^X, \nu^X)\).
	Thanks to \cite[Proposition II.2.9]{JS} we may w.l.o.g. assume that
	\begin{align}\label{a leq 1}
	\{a \leq 1\} = \Omega \times [0, T],\ \textup{ where }\ a_t \triangleq \nu^X(\{t\}\times\mathbb{R}).
	\end{align}
	The stock price process \(S\) of an \(\mathscr{H}\)-SII model is given~by 
	\begin{align*}
	S_t \triangleq e^{X_t},\ \ t \in [0, T]. 
	\end{align*}
	Clearly, the assumption \(S_0 = 1\) is no restriction and serves only the purpose of notational convenience. 
	Let us now define our key objects of interest.
	
	\begin{definition}\label{def M}
		We denote by \(\M\) the set of \emph{structure preserving equivalent martingale measures}, i.e. all probability measure \(\q\) on \((\Omega, \mathscr{F})\) such that the following holds:
		\begin{enumerate}
			\item[\textup{(i)}]
			\(\q \sim \p\).
			\item[\textup{(ii)}]\(S\) is an \((\F, \q)\)-martingale.
			\item[\textup{(iii)}] \(X\) is an \((\mathscr{H}, \F, \q)\)-SII. 
			\item[\textup{(iv)}] The \((\F, \q)\)- and \((\G, \q)\)-characteristics of \(X\) coincide.
		\end{enumerate}
	\end{definition}
	In our setting we do not need to distinguish between structure preserving equivalent \emph{true, local} or \emph{sigma} martingale measures, since all exponential \(\mathscr{H}\)-SIIs which are sigma martingales are martingales, cf. Lemma \ref{eq coro} below.

	\begin{definition}\label{def Y}
		We define \(\mathscr{Y}\) to be the set of all tuple \((\beta, U)\) which satisfy the following: \(\beta\) is a real-valued
		\(\F\)-predictable process and \(U\) is a \([0, \infty)\)-valued \(\mathscr{P}(\F)\otimes\mathscr{B}\)-measurable function such that
		\begin{enumerate}
			\item[\textup{(i)}]
			\(\{U > 0\} = \{a' \leq 1\} = \Omega \times [0, T]\) and \(\{a = 1\} = \{a' = 1\}\), where
			\(a'_t \triangleq \int_{\mathbb{R}} U(t, x) \nu^X(\{t\}\times \dd x).\)
			\item[\textup{(ii)}]
			\(\p\)-a.s. it holds that \(|h(x)(U - 1)| \star \nu^X_T < \infty\) and
			\begin{align}\label{H}
			H_T \triangleq\beta^2 \cdot &\hspace{0.04cm} C^X_T + \left(1 - \sqrt{U}\hspace{0.04cm} \right)^2 \star \nu^X_T + \sum_{s \in [0, T]} \left( \sqrt{1 - a_s} - \sqrt{1 - a_s'}\hspace{0.04cm} \right)^2 < \infty. 
			\end{align}
			\item[\textup{(iii)}] 
			\(\p\)-a.s. it holds that \((e^{x} - 1) U \1_{\{x > 1\}} \star \nu^X_T < \infty\) and that for all \(t \in [0, T]\)
			\begin{equation}\label{MPRE}
			\begin{split}
			B^X_t &+  \bigg(\beta+ \frac{1}{2} \bigg) \cdot C^X_t + \big((e^{x} - 1)U -h(x)\big) \star \nu^X_t
			\\&+\sum_{s \in [0, t]} \left(\log(1 + \widetilde{V}_s) - \widetilde{V}_s\right) = 0,
			\end{split}
			\end{equation}
			where \(\widetilde{V}_t \triangleq \int_\mathbb{R} (e^{x} - 1) U(t, x) \nu^X(\{t\}\times\dd x)\).
			\item[\textup{(iv)}]
			the modified characteristics
			\begin{equation}\label{version H mb}
			\begin{split}
			B\triangleq B^X  +\beta \cdot C^X + h(x)(U - 1) \star \nu^X,\qquad
			C \triangleq C^X,\qquad
			\nu \triangleq U\cdot \nu^X,
			\end{split}
			\end{equation}
			have a \(P\)-version which is \(\mathscr{H}\)-measurable.
		\end{enumerate}
	\end{definition}
	Motivated by Girsanov's theorem \cite[Theorem III.3.24]{JS}, the elements in \(\mathscr{Y}\) are called \emph{Girsanov quantities}.
	The function \(U\) is used to influence the jump structure of \(X\) and both \(U\) and \(\beta\) change the drift of \(X\).
	If \(U\) is given by
	\(
	U(t, x) = \frac{e^{\beta_t x}}{1 + \widehat{W}_t}\), where \(\widehat{W}_t \triangleq \int_\mathbb{R} \left(e^{\beta_t x} - 1\right) \nu^X(\{t\} \times \dd x),\)
	then \((\beta, U)\) correspond to the famous \emph{Esscher measure}, cf., e.g., \cite{KS(2002b)}.
	The equation \eqref{MPRE} is often called \emph{market price of risk equation (MPRE)}.

	The set \(\{a > 0\}\) is thin. Consequently, as a section of a thin set, \(\{t \in [0, T]\colon a_t(\omega) > 0\}\) is at most countable and the sums in Definition \ref{def Y} (ii) and (iii) are well-defined. 
	Part (ii) of Definition \ref{def Y} implies that \(B \in \mathscr{V}(\F, \p)\).
	Note that
	\begin{align*}
	(1 \wedge |x|^2) U \star \nu^X_t \leq 4 (1 \wedge |x|^2) \star \nu^X_t + 4 \left(1 - \sqrt{U}\right)^2 \star \nu^X_t.
	\end{align*}
	Hence, \(U \cdot \nu^X\) makes sense as a candidate for a compensator.

	Now, we are in the position to state our main result, which generalizes \cite[Proposition 1]{EJ97} to \(\mathscr{H}\)-SII models.
	For a detailed proof we refer to Section \ref{Proofs} below.
	\begin{theorem}\label{MT2}
		We have
		\begin{align*}
		\mathscr{Y} \not = \emptyset\hspace{0.1cm} \Longleftrightarrow\hspace{0.1cm} \M \not = \emptyset.
		\end{align*}
		Moreover, the following holds:
		\begin{enumerate}
			\item[\textup{(i)}] For each \((\beta, U)\in \mathscr{Y}\) there exists a \(\q\in \M\) such that the \((\F, \q)\)- and \((\G, \q)\)-characteristics of \(X\) are given by \eqref{version H mb}.
			\item[\textup{(ii)}]
			For each \(\q \in \M\) there exists a pair \((\beta, U) \in \mathscr{Y}\) such that \(X\) has \((\F, \q)\)- and \((\G, \q)\)-characteristics given by \eqref{version H mb}.
		\end{enumerate}
	\end{theorem}
	We stress that the integrability assumptions in the definition of \(\mathscr{Y}\) have an almost sure character in contrast to classical exponential moment conditions of Novikov-type, which are typically imposed to guarantee the existence of an EMM.
	
	We give a short outline of the proof of Theorem \ref{MT2}.
	For given Girsanov quantities \((\beta, U)\) we may define the candidate density process 
	\begin{align}\label{ZK}
	Z \triangleq \mathscr{E}\bigg(\beta \cdot X^{c} + \bigg\{U - 1 + \frac{a' - a}{1 - a}\1_{\{a < 1\}}\bigg\} \star \big(\mu^X - \nu^{X}\big)\bigg).
	\end{align}
	We show in Lemma \ref{lemma Z} below that (i), (ii) and (iv) of Definition \ref{def Y} imply that \(Z\) is a positive martingale. 
	The proof is based on the observation that \(\log Z\) is an \(\mathscr{H}\)-SII and that exponential \(\mathscr{H}\)-SIIs are martingales if they are local martingales.
	Now, we define a candidate measure \(\q\) for \(\M\) by \(\q(G) = E_P[Z_T \1_G]\) for \(G \in \mathscr{F}\). 
	
	On the infinite time horizon \(Z\) may not be a uniformly integrable martingale. This is the only point where the proof of sufficiency depends on the finite time horizon. 
	However, when we consider an infinite time horizon, we can define the consistent family \((\mathscr{F}_t, \q_t)_{t \geq 0}\) by \(Q_t(G) = E^P[Z_t \1_G]\) for \(G \in \mathscr{F}_t\).
	Now, if the filtered space \((\Omega, \mathscr{F}, \F)\) allows extensions of consistent families, there exists a probability measure \(\q\) on \((\Omega, \mathscr{F})\) such that \(Q = Q_t\) on \(\mathscr{F}_t\) for all \(t \in [0, \infty)\). Since \(Z\) is positive, this implies that \(Q\) and \(P\) are locally equivalent. 
	These considerations lead to a \emph{local} version of Theorem \ref{MT2}. We stress that classical path spaces allow the extension of consistent families, cf. \cite[Proposition 3.9.17]{Bichteler02}.
	
	Checking that \(Q\) satisfies (ii) and (iii) of Definition \ref{def M} is identical for the finite and the infinite time horizon.
	
	Let us also comment on the converse direction.
	If \(\q\in \M\), then Girsanov's theorem yields the existence of candidate Girsanov quantities \((\beta, U)\). The integrability conditions follow from general results on absolute continuity of laws of semimartingales and the equivalence of \(\p\) and \(\q\) allows a modification of \((\beta, U)\) such that Definition \ref{def Y} (i) is satisfied. These results can be applied irrespective of a finite or an infinite time horizon.

	\subsection{\(\mathscr{H}\)-SII Models with Continuous Local Martingale Part}\label{The Case of Local Characteristics}
	Let us shortly discuss a simplified situation
	in which \(X\) has a non-trivial continuous local martingale part. 
	
	It is well-known that the \((\F, \p)\)-characteristics of \(X\) have a decomposition 
	\begin{align*}
	B^X = b^X \cdot A^X,\quad C^X = c^X \cdot A^X,\quad \nu^X = F^X \otimes A^X, 
	\end{align*}
	cf. \cite[Proposition II.2.9]{JS}.
	Here, \(A^X\) is an \(\F\)-predictable process in \(\mathscr{A}^+_{\textup{loc}}(\F, \p)\), \(b^X\) is an \(\F\)-predictable process, \(c^X\) is an \(\F\)-predictable non-negative process and \(F^X_{\omega, t}(\dd x)\) is a transition kernel from \((\Omega \times [0, T], \mathscr{P}(\F))\) to \((\mathbb{R}, \mathscr{B})\).
	We call \((b^X, c^X, F^X; A^X)\) \emph{local \((\F, \p)\)-characteristics of \(X\)}. 
	Thanks to Standing Assumption \ref{SA}, \((b^X, c^X, F^X; A^X)\) are also local \((\G, \p)\)-characteristics of \(X\).
	
	\begin{corollary}\label{concrete coro}
		Suppose that \(\nu^X(\{t\} \times \mathbb{R}) = 0\) for all \(t \in [0, T]\), that
		\(c^{X} \not = 0\) and that there exists an \(\mathscr{H}\otimes \mathscr{B}([0, T]) \otimes \mathscr{B}\)- and \(\mathscr{P}(\F)\otimes\mathscr{B}\)-measurable positive function \(U\) such that \(P\)-a.s.
		\begin{equation*} 
		\begin{split}
		|h(x) (U - 1)|\star \nu^X_T &+ |e^x - 1|U\1_{\{x > 1\}}\star \nu^X_T
		+ \beta^2 \cdot C^X_T+ \left(1 - \sqrt{U}\right)^2 \star \nu^X_T < \infty,
		\end{split}
		\end{equation*}
		where
		\begin{align}\label{beta explicit}
		\beta 
		\triangleq
		- \frac{1}{c^{X}} \left( \frac{1}{2} c^{X} + b^{X} + \int_{\mathbb{R}} \big((e^x - 1)U(\cdot, x) - h(x)\big) F^{X}(\dd x)\right),
		\end{align}
		then \(\M \not = \emptyset\).
		Moreover, there exists a \(\q \in \M\) such that \(X\) has local \((\F, \q)\)- and local \((\G, \q)\)-characteristics \((b^{\hspace{0.0cm}\q, X}, c^{X}, F^{\q,X}; A^{X})\), where
		\begin{align*} 
		b^{\hspace{0.0cm}\q, X} = -\left( \frac{1}{2} c^{X} + \int_{\mathbb{R}} \big(e^x - 1 - h(x)\big)U(\cdot, x)F^{X}(\dd x)\right) 
		\end{align*}
		and \(F^{\q, X}(\dd x) = U(\cdot, x) F^X(\dd x)\).
	\end{corollary}
	\begin{proof}
		If \((\beta, U) \in \mathscr{Y}\), then Theorem \ref{MT2} implies the claims.
		It is assumed that \((\beta, U)\) satisfies (i), (ii) and (iii) of Definition \ref{def Y}. Moreover, \((\beta, U)\) also satisfies (iv) thanks to Lemma \ref{ML} (i) in Appendix \ref{MLemmata A}. This concludes the proof.
		~\end{proof}
	If we may choose \(U = 1\) there exists a measure \(\q\in \M\) which does not change the jump structure of \(X\).

	\section{Examples}\label{CS}
	
	In this section we discuss two examples.
	Firstly, we investigate a generalization of the Nobel Prize winning model of Black and Scholes \cite{BlackScholes73}, 
	introducing an additional independent stochastic factor, which for instance may be a fractional Brownian motion.
	Secondly, we consider a time-changed L\'evy model as introduced by Carr, Geman, Madan and Yor \cite{CGMY03}. 
	
	\subsection{A Generalized Black-Scholes Model with Independent Factor}\label{A Generalized Black-Scholes Model with Independent Factor}
	Let \(\F, \mathscr{H}\), \(Y\) and \(V \triangleq (I, W)\) be as in
	Example \ref{Independent Integrands} in Appendix \ref{scope}. 
	Here, we denote \(I_t = t\).
	Then Standing Assumption \ref{SA1} holds.
	We assume that \(W\) is a Brownian motion which is \(\p\)-independent of \(Y\).
	Moreover, let \(\gamma \colon \mathbb{D}^m \times [0, T] \to \mathbb{R}\) and \(\sigma \colon \mathbb{D}^m \times [0, T] \to (0, \infty)\) be such that \(\gamma(Y), \sigma(Y)\) are \(\F^Y\)-predictable and 
	\begin{align}\label{well-defined}
	\p\left(\int_0^T |\gamma(Y, s)| \dd s + \int_0^T\sigma(Y, s)^2\dd s < \infty\right) = 1. 
	\end{align}
	We now set
	\begin{align*}
	X \triangleq \int_0^\cdot \gamma(Y, s) \dd s + \sigma(Y) \cdot W.
	\end{align*}
	Standing Assumption~\ref{SA} holds thanks to Corollary \ref{SA Coro} in Appendix \ref{scope}.
	We obtain very mild sufficient and necessary conditions for \(\M\not = \emptyset\).
	\begin{corollary}\label{coro BSISF}
		\(\M\not = \emptyset\) if, and only if, 
		\begin{align}\label{besser als novi}
		\p\left(\int_0^T \bigg(\frac{\gamma(Y, s)}{\sigma(Y, s)}\bigg)^2\dd s< \infty\right ) = 1.
		\end{align}
	\end{corollary}
	\begin{proof}
		The implication \(\Longleftarrow\) follows from Corollary \ref{concrete coro} and \eqref{well-defined}. If \(\M\not = \emptyset\), then Theorem \ref{MT2} yields that the MPRE \eqref{MPRE} has a solution \(\beta\) such that \(\p\)-a.s. \(\beta^2 \cdot C^X_T < \infty\).
		We obtain that \(\beta \sigma^2(Y) = - \gamma(Y) - \sigma^2(Y)/2\) up to a \(\p \otimes\dd t\)-null set. Hence, we deduce \eqref{besser als novi} from \eqref{well-defined} together with \(\p\)-a.s. \(\beta^2 \cdot C^X_T < \infty\). 
	\end{proof}
	
	\subsection{CGMY-Model with Independent Stochastic Volatility}
	We pose ourselves in the setting introduced in Example \ref{Time-changed Levy processes} in Appendix \ref{scope}. 
	Let \(Y\) be an Ornstein-Uhlenbeck process driven by a L\'evy subordinator \(L\) with constant initial value \(Y_0 > 0\) 
	and parameter \(\lambda > 0\). More precisely, we assume that \[Y_t \triangleq Y_0 e^{- \lambda t} + e^{- \lambda(t - s)} \cdot L_t,\quad t \in [0, T].\]
	From this definition we immediately deduce that for all \(t \in [0, T]\)
	\begin{align}\label{Y bound}
	Y_t \geq Y_0 e^{- \lambda t} \geq Y_0 e^{- \lambda T} > 0. 
	\end{align}
	Let \(V\) be a one-dimensional L\'evy process with L\'evy-Khinchine triplet \((b^{V}, c^{V}, F^{V})\) and that \(h(x) = x \1_{\{|x| \leq 1\}}\).
	Then, we set 
	\begin{align*}
	X_t \triangleq\mu t + V_{\int_0^t Y_{s-} \dd s},\quad t \in [0, T].
	\end{align*}
	Note that both Standing Assumptions \ref{SA1} and \ref{SA} hold.
	
	\begin{proposition}
		Assume that \(\int_{\mathbb{R}} (1 \wedge |x|)F^V(\dd x) < \infty\).
		Then \(\M \not = \emptyset\) if at least one of the following conditions hold:
		\begin{enumerate}
			\item[\textup{(i)}] \(c^V \not = 0\). 
			\item[\textup{(ii)}]
			\(F^V((- \infty, -1)) > 0\) and \(F^V((1, \infty)) > 0\).
		\end{enumerate}
	\end{proposition}
	
	\begin{proof}
		\textbf{(i).}
		The local characteristics of \(X\) are given as in Lemma \ref{remark TC} in Appendix \ref{scope}.
		We choose 
		\begin{align*}
		U(\omega, t, x) &\triangleq\ \frac{1}{1 - e^x} \1_{\{x < -1\}} + \frac{x}{e^x - 1} \1_{\{|x| \leq 1\} \backslash \{0\}} + \1_{\{0\}}+ \frac{1}{e^x - 1} \1_{\{x > 1\}},
		\\
		\beta_t &\triangleq\ \frac{- \mu}{c^{V} Y_{t-}} - \frac{1}{c^V} \left(b^{V} + \int_{\mathbb{R}} \left(\1_{\{x > 1\}} - \1_{\{x < -1\}}\right) F^{V}(\dd x)\right)- \frac{1}{2}. 
		\end{align*}
		Obviously, \(U\) is positive and \(\mathscr{H}\otimes\mathscr{B}([0,T|)\otimes\mathscr{B}\)- and \(\mathscr{P}(\F)\otimes\mathscr{B}\)-measurable. 
		Taylor's theorem yields the existence of a non-negative constant \(K\) such that 
		\begin{align*}
		|h(x)(U(x) - 1)| + \left(1 - \sqrt{U(x)}\right)^2 +(e^x - 1)U\1_{\{x >1\}}  \leq K (1 \wedge |x|). 
		\end{align*}
		The assumption that \(1 \wedge |x|\) is \(F^V\)-integrable 
		and the bound \eqref{Y bound} yield that the integrability condition of Corollary \ref{concrete coro}, and hence the claim, holds. 
		
		\textbf{(ii).}
		W.l.o.g. we may assume that \(c^{V} = 0\). 
		We set \(\beta \triangleq 0\) and
		\begin{align*}
		U(t, x) \triangleq\ \ &\left(\frac{b^{V} \1_{\{b^{V} \geq 0\}} + Y^{-1}_{t-} \mu \1_{\{\mu \geq 0\}}}{(1 - e^x)F^{V}((-\infty, -1))}+ \frac{F^V((1, \infty))}{1 - e^x}\right)\1_{\{x < -1\}} 
		\\+ &\left(\frac{-b^{V} \1_{\{b^{V} < 0\}} - Y^{-1}_{t-}\mu \1_{\{\mu< 0\}}}{(e^x - 1) F^{V}((1, \infty))} + \frac{F^V((- \infty, -1))}{e^x - 1}\right) \1_{\{x > 1\}}
		\\+ &\ \frac{x}{e^x - 1} \1_{\{|x| \leq 1\}\backslash \{0\}} + \1_{\{0\}}.
		\end{align*}
		Part (i) of Definition \ref{def Y} trivially holds and it is routine to check that (iii) is satisfied.
		Part (ii) follows by Taylor's theorem as above. 
		Since \(U\) is \(\mathscr{H}\otimes\mathscr{B}([0, T])\otimes\mathscr{B}\)-measurable, Lemma \ref{ML} in Appendix \ref{MLemmata A} yields part (iv).
		Hence \((0, U) \in \mathscr{Y}\) and the claim follows from Theorem~\ref{MT2}. 
	\end{proof}
	\section{Proof of Theorem \ref{MT2}}\label{Proofs}
	\subsection{Martingale Property of Exponential \(\mathscr{H}\)-SII Processes}\label{Martingality of exponential SCII processes}
	The following lemma generalizes \cite[Lemma A.1]{KMK10} to arbitrary \(\mathscr{H}\)-SIIs. 
	
	\begin{lemma}\label{eq coro}
		Let \(Y\) be an \((\mathscr{H}, \F, \p)\)-SII with \((\G, \p)\)-character\-istics \((B^Y, C^Y, \nu^Y)\).
		\begin{enumerate}
			\item[\textup{(i)}] 
			The following are equivalent:
			\begin{enumerate}
				\item[\textup{(I)}]
				\(e^Y\) is an \((\G, \p)\)-martingale.
				\item[\textup{(II)}]
				\(e^Y\) is a local \((\G, \p)\)-martingale.
				\item[\textup{(III)}]
				\(e^Y\) is a sigma \((\G, \p)\)-martingale.
				\item[\textup{(IV)}] We have 
				\(
				e^x \1_{\{x > 1\}} \star \nu^Y \in \mathscr{V}(\G, \p)
				\)
				and \(\p\)-a.s.
				\begin{equation}\label{MPR}
				\begin{split}
				B^{Y}+ \frac{1}{2} C^{Y}  &+ \big(e^{x} - 1 - h(x)\big) \star \nu^{Y}
				+ \sum_{s \in [0, \cdot]}\left(\log(1 + \widehat{Y}_s) - \widehat{Y}_s\right) = 0,
				\end{split}
				\end{equation}
				where \(\widehat{Y}_t \triangleq \int_\mathbb{R} (e^{x} - 1) \nu^{Y}(\{t\} \times \dd x)\). 
			\end{enumerate}
			\item[\textup{(ii)}]
			In addition, if \(Y\) is an \((\F, \p)\)-semimartingale and its \((\F, \p)\)-character\-istics coincide with \((B^Y, C^Y, \nu^Y)\), then \textup{(I) \(\Longleftrightarrow\) (II) \(\Longleftrightarrow\) (III) \(\Longleftrightarrow\) (IV)} where \textup{(I) - (IV)} are given as in \textup{(i)} with \(\G\) replaced by \(\F\).
		\end{enumerate}
	\end{lemma}
	
	\begin{proof}
		\textbf{(i).} 
		The implication (I) \(\Longrightarrow\) (II) is trivial and the implication (II) \(\Longrightarrow\) (III) holds due to \cite[Proposition III.6.34]{JS}.
		The implication (III) \(\Longrightarrow\) (II) follows from the fact that non-negative sigma martingales are local martingales, cf. \cite[p.~216]{JS}.
		An exponential semimartingale is a local martingale if, and only if, it is exponentially special and its exponential compensator vanishes, cf. \cite[Lemma 2.15]{KS(2002b)}.
		Hence, the equivalence \textup{(II)} \(\Longleftrightarrow\) \textup{(IV)} follows from \cite[Lemma 2.13, Theorem 2.18, Theorem 2.19]{KS(2002b)}.
		It is left to prove the implication (IV) \(\Longrightarrow\) (I).
		Thanks to the equivalence (III) \(\Longleftrightarrow\) (IV) and \cite[Proposition 3.1]{doi:10.1137/S0040585X980312}, the process
		\(e^Y\) is a non-negative \((\G, \p)\)-supermartingale. Thus, we have to show that \(\E_\p[e^{Y_T}] = 1\).
		Thanks to Remark \ref{Remark SA} (ii), Lemma \ref{lemma kernel} and \cite[Lemma 2.13, Theorem 2.18, Theorem 2.19]{KS(2002b)} \(\p\)-a.s. the process \(Y\) is a \((\{\Omega, \emptyset\}, \G, P(\cdot |\mathscr{H}))\)-SII and \(e^Y\) is \(P\)-a.s. a local \((\G, P(\cdot|\mathscr{H}))\)-martingale.  
		Hence, using \cite[Proposition 3.12]{KMK}, the process \(e^Y\) is \(P\)-a.s. a \((\G, P(\cdot|\mathscr{H}))\)-martingale. This implies that \(P\)-a.s.
		\(
		\int_\Omega e^{Y_T(\omega')} P(\dd \omega' |\mathscr{H}) = 1.
		\)
		Taking \(\p\)-expectation 
		finishes the proof.
		
		\textbf{(ii).} 
		(I) \(\Longrightarrow\) (II) \(\Longleftrightarrow\) (III) \(\Longleftrightarrow\) (IV) follow as in (i).
		Thanks to (i), (IV) implies \(\E_\p[e^{Y_T}] = 1\). Hence, we can also conclude (IV) \(\Longrightarrow\) (I).
	\end{proof}
	
	\subsection{A Candidate Density Process}\label{A Candidate Density Process}
	Standing Assumption \ref{SA} and \cite[Theorem II.2.34]{JS} imply that the continuous local \((\F, \p)\)- and \((\G, \p)\)-martin\-gale parts of \(X\) coincide. We denote them by~\(X^c\).
	\begin{lemma}\label{lemma Z}
		Let \((\beta,U)\in \mathscr{Y}\), then the process \(Z\) as given by \eqref{ZK} 
		is a positive \((\F, \p)\)- and \((\G, \p)\)-martingale.
	\end{lemma}
	\begin{proof}
		Let us start by showing that \(Z\) is a positive local \((\F, \p)\)- and \((\G, \p)\)-martin\-gale. 
		We have \(\F \subseteq \G\), i.e. \(\mathscr{F}_t \subseteq \mathscr{G}_t\) for all \(t \in [0, T]\).
		Since \(\beta\) is \(\F\)-predictable, \(\F \subseteq \G\) implies that \(\beta\) is \(\G\)-predictable. 
		Now, \(\p\)-a.s. \(\beta^2 \cdot C^X_T < \infty\) yields that
		\(\beta \cdot X^c\) is a local \((\F, \p)\)- and \((\G, \p)\)-martingale.
		We denote 
		\begin{align*}
		V(t, x) &\triangleq U(t, x) - 1 + \frac{a'_t - a_t}{1 - a_t}\1_{\{a_t < 1\}},\\
		\widehat{V}(t, x) &\triangleq \int_\mathbb{R} V(t, x) \nu^X(\{t\}\times \dd x) 
		= \frac{a'_t- a_t}{1 - a_t} \1_{\{a_t < 1\}},
		\end{align*}
		where we use that \(\{a= 1\} = \{a' = 1\}\). 
		Recalling \(\F\subseteq \G\), we obtain that \(V\) is \(\mathscr{P}(\F)\otimes\mathscr{B}\)- and \(\mathscr{P}(\G)\otimes\mathscr{B}\)-measurable. 
		Moreover, we have
		\begin{equation}\label{jump N}
		\begin{split}
		\widetilde{V}_t &\triangleq V(t, \Delta X_t) \1_{\{\Delta X_t \not = 0\}} - \widehat{V}_t = \begin{cases}
		U(t, \Delta X_t) - 1,&\textup{on } \{\Delta X_t \not = 0\},\\
		-\frac{a'_t - a_t}{1 - a_t} \1_{\{a_t < 1\}},&\textup{on } \{\Delta X_t = 0\}.
		\end{cases}
		\end{split}
		\end{equation}
		Since \(\{U > 0\} = \Omega \times [0, T]\)
		and \(\{a' < 1\} = \{a < 1\}\),
		we have 
		\(\{\widetilde{V}> -1\} = \Omega \times [0, T]\). 
		Now, 
		\cite[Theorem II.1.33 d)]{JS} yields that \(V \star (\mu^X - \nu^X)\) is a local \((\F, \p)\)- and \((\G, \p)\)-martingale
		if \(\p\)-a.s.
		\begin{align*}
		K_T \triangleq \left(1 + \sqrt{1 + V - \widehat{V}}\ \right)^2\star\ & \nu^X_T + \sum_{s \in [0, T]} \big(1 - a_s\big) \left(1 - \sqrt{1 - \widehat{V}^2_s}\ \right)^2 < \infty. 
		\end{align*}
		This holds since \(K_T \leq H_T\) and \(\p\)-a.s. \(H_T < \infty\), cf. \eqref{H} for the definition of \(H_T\).
		Therefore, \(Z\) is a local \((\F, \p)\)- and \((\G, \p)\)-martingale, which is positive due to the fact that \(\{\widetilde{V} > -1\} = \Omega \times [0, T]\) together with \cite[Theorem I.4.61 c)]{JS}. 
		
		Using Lemma~\ref{eq coro}, \(Z\) is a \((\G, P)\)-martingale if \(\log Z\) is an \((\mathscr{H}, \F, P)\)-SII. 
		Since \(Z\) is an \((\F, \p)\)-supermartingale, this also yields that \(Z\) is an \((\F, \p)\)-martingale. 
		We proceed in two steps: Firstly, we compute the \((\G, P)\)-characteristics of \(\log Z\). Secondly, we show that they have \(\mathscr{H}\)-measurable \(P\)-versions.
		We define the local \((\G, \p)\)-martingale
		\(
		N\triangleq \beta \cdot X^{c} + V\star \big(\mu^X - \nu^X\big)
		\)
		and denote its \((\G, \p)\)-characteristics by \((B^N, C^N, \nu^N)\).
		The continuous local \((\G, \p)\)-martin\-gale part of \(N\) is given by \(\beta \cdot X^c\) and hence \( 
		C^N = \beta^2 \cdot C^X.\) 
		Similarly as in \cite{KLS}, it follows that the \((\G, \p)\)-compensator \(\nu^N\) of \(\mu^N\) is given by
		\begin{equation}\label{nuN}
		\begin{split}
		\1_G& \star \nu^{N}
		= \1_G\big(U - 1\big)\star \nu^X + \sum_{t \in [0, \cdot]}\1_{\{a_t> 0\}} \1_G\bigg(- \frac{a'_t - a_t}{1 - a_t}\bigg) (1 - a_t),
		\end{split}
		\end{equation}
		for \(G \in \mathscr{B}, 0 \not \in G\).
		Since \(N\) is a local \((\G, \p)\)-martingale, \cite[Proposition II.2.29]{JS} yields that
		\(
		B^{N}(h') = - (x - h'(x)) \star \nu^{N}.
		\)
		Since identically \(\Delta N = \widetilde{V} > -1\), \cite[Theorem II.8.10]{JS} yields that 
		\(\log Z\) has \((\G, \p)\)-characteristics given by 
		\begin{equation}\label{log Z chara}
		\begin{split}
		&B^{\log Z} = B^{N} - \frac{1}{2} C^{N} + \big(h (\log(1 + x)) - h(x)\big) \star \nu^{N},
		\\
		C^{\log Z} &= C^{N},
		\quad
		\1_A \star \nu^{\log Z} =  \1_A(\log(1 + x))\star \nu^{N},\quad A \in \mathscr{B}, 0 \not \in A.
		\end{split}
		\end{equation}
		Since \(\nu^X\) and \(U \cdot \nu^X\) have \(\mathscr{H}\)-measurable \(\p\)-versions and \(P\)-a.s. \(|h(x)(U - 1)| \star \nu^X_T < \infty\), Lemma \ref{ML} (i) in Appendix \ref{MLemmata A} yields that 
		\(h(x)(U - 1) \star \nu^X\) has an \(\mathscr{H}\)-measurable \(\p\)-version. Hence, since \(B^X\) and 
		\(
		B^X + \beta \cdot C^X + h(x)(U - 1) \star \nu^X
		\)
		have \(\mathscr{H}\)-measurable \(\p\)-versions, so does \(\beta \cdot C^X\). 
		Now Lemma \ref{ML2} in Appendix \ref{MLemmata A} implies that \(C^N\) has an \(\mathscr{H}\)-measurable \(\p\)-version.
		Recalling \eqref{nuN}, Lemma \ref{ML} (i) and (ii) yield that \(\nu^N\) has an \(\mathscr{H}\)-measurable \(\p\)-version.
		For all \(g \in \mathscr{I}\) we have 
		\(
		|g(\log(1 + x))| 
		\leq 3 (1 \wedge |x|^2)
		\).
		Moreover, \(P\)-a.s. \((|x - h'(x)| + |h(\log(1 + x)) - h (x)|) \star \nu^N_T < \infty\) follows from \(P\)-a.s. \((|x| \wedge |x|^2) \star \nu^N_T < \infty\), cf. \cite[Proposition II.2.29]{JS}.
		Therefore, since \(\nu^N\) has an \(\mathscr{H}\)-measurable \(\p\)-version, 
		Lemma \ref{ML} (i) implies that \(B^{\log Z}(h)\) and \(\nu^{\log Z}\) also have \(\mathscr{H}\)-measurable \(\p\)-versions. This concludes the proof.
	\end{proof}
	\begin{remark}
		The statement of Lemma \ref{lemma Z} stays true if the pair \((\beta, U)\) only satisfies (i), (ii) and (iv) in Definition \ref{def Y}. 
	\end{remark}

	\subsection{Proof of Theorem \ref{MT2}}\label{proof1}
	Let \((\beta, U) \in \mathscr{Y}\) and \(Z\) as in \eqref{ZK}. 
	Thanks to Lemma \ref{lemma Z}, \(Z\) is a positive \((\F, \p)\)- and \((\G, \p)\)-martin\-gale.
	Define \(Q\) by \(\q (A) = E^P[Z_T\1_A]\) for \(A \in \mathscr{F}\). Since \(P\)-a.s. \(Z_T > 0\), it holds that \(Q\sim P\).
	Thanks to the martingale property of \(Z\), \(Q(A) = E^P[Z_t \1_A]\) for \(A \in \mathscr{G}_t\) and \(t \in [0, T]\).
	From \eqref{jump N} it follows that \(P\)-a.s.
	\(
	\langle Z^c, X^c\rangle^{\p, \K} = Z_-\beta \cdot C^X\) and
	\(M^{\p}_{\mu^{X}}(Z |\mathscr{P}(\K) \otimes \mathscr{B})  = Z_{-} U\ \textup{ for }\ \K \in \{\F, \G\}.\)
	Now, using Girsanov's theorem \cite[Theorem III.3.24]{JS}, \(X\) is an \((\F, \q)\)- and \((\G, \q)\)-semimartingale with \((\F, \q)\)- and \((\G, \q)\)-characteristics given by \eqref{version H mb}. Since \((\beta, U) \in \mathscr{Y}\) and \(Q \sim P\), these characteristics have an \(\mathscr{H}\)-measurable \(\q\)-version, i.e.
	\(X\) is an \((\mathscr{H}, \F, \q)\)-SII. Moreover, since the MPRE \eqref{MPRE} holds, \(S\) is 
	an \((\F, \q)\)-martingale by Lemma \ref{eq coro}. 
	Therefore, we have shown that (i) holds and \(\mathscr{Y}\not = \emptyset \Longrightarrow \M \not = \emptyset\).

	Next, we prove (ii) and \(\M \not = \emptyset \Longrightarrow \mathscr{Y} \not = \emptyset\).
	Take \(\q\in \M\) and denote the \(\F\)-density process of \(\q\) w.r.t. \(\p\) by \(Z^*\).
	For all \(t \in [0, T]\) we have \(\q(Z^*_t = 0) = \E_\p[Z^*_t \1_{\{Z^*_t = 0\}}] = 0\).
	In view of \cite[Proposition I.2.4, Theorem III.3.4]{JS}, we also have \(\q(Z^*_{t-} = 0) = \E^\p[Z^*_{t-}\1_{\{Z^*_{t-} = 0\}}] = 0\). Hence, \(Q \sim P\) and \cite[Lemma III.3.6]{JS} yields 
	\(\p\)-a.s. \(Z^*_t > 0\) and \(Z^*_{t-} > 0\) for all \(t \in [0, T]\). 
	Denote \(\Lambda \triangleq \{M^\p_{\mu^X}(Z^*|\mathscr{P}(\F)\otimes\mathscr{B}) > 0\} \cap \{Z^*_- > 0\} \times \mathbb{R}\)
	and
	\begin{align*}
	U^*(&\omega, t, x)\triangleq \begin{cases}
	\frac{1}{Z^*_{t-}(\omega)}M^\p_{\mu^X}(Z^*|\mathscr{P}(\F)\otimes\mathscr{B})(\omega, t, x),&\hspace{-0.25cm}\textup{on } \Lambda, \\
	1,&\hspace{-0.25cm}\textup{otherwise.} 
	\end{cases}
	\end{align*}
	Girsanov's theorem \cite[Theorem III.3.24]{JS} yields the existence of an \(\F\)-predictable process \(\beta\) such that 
	the \((\F, \q)\)-characteristics of \(X\) are given by \eqref{version H mb} with \(U\) replaced by \(U^*\). 
	Moreover, since \(X\) is an \((\mathscr{H}, \F, \q)\)-SII and its \((\F, \q)\)-characteristics coincide with its \((\G, \q)\)-character\-istics, 
	there exists an \(\mathscr{H}\)-measurable \(\q\)-version of these characteristics. 
	Since \(\p \sim \q\), there also exists an \(\mathscr{H}\)-measurable \(\p\)-version.
	Using again Girsanov's theorem and \(Q \sim P\), we obtain that \(\p\)-a.s. \(|h(x)(U^* - 1)| \star \nu^X_T < \infty\).
	Since \(e^X\) is an \((\F, \q)\)-martingale and \(Q \sim P\), \cite[Lemma 2.13, Theorem 2.19]{KS(2002b)} imply that \(\p\)-a.s. \(
	(e^x - 1)\1_{\{x > 1\}}U^* \star \nu^X_T
	< \infty
	\) and that the MPRE \eqref{MPRE} holds \(\p\)-a.s. for all \(t \in [0, T]\) with \(U\) replayed by \(U^*\).
	Denote by \(H^*\) the process \(H\) with \(U\) replaced by \(U^*\). Since \(\q \sim \p\), we deduce from \cite[Theorem \(1^{**}\)]{KLS-LACOM1}, that 
	\(\p\)-a.s. \(H^*_T < \infty\). 
	We now show that there exists a \(\mathscr{P}(\F)\otimes \mathscr{B}\)-measurable function \(U\) and a \(\p\)-evanescence set \(\Lambda'\) such that \(U = U^*\) on \(\complement \Lambda' \times \mathbb{R}\), \(\{U > 0\} = \{a' \leq 1\} = \Omega \times [0, T]\) and \(\{a = 1\} = \{a' = 1\}\).
	The properties of \(U^*\) then readily extend to \(U\) and \((\beta, U) \in \mathscr{Y}\) follows. 
	Denote \(a^*_t \triangleq (U^* \cdot \nu^X)(\{t\} \times \mathbb{R})\).
	Due to the fact that subsets of thin sets are itself thin, cf. \cite[Theorem 3.19]{HWY}, the set \(\{a = 1\} \subseteq \{a > 0\}\) is thin. Hence, \cite[Lemma I.2.23]{JS} yields the existence of a sequence of \(\F\)-predictable times \((\tau_n)_{n \in \mathbb{N}}\) such that \(\{a = 1\} = \bigcup_{n \in \mathbb{N}} \of \tau_n\gs\) up to \(\p\)-evanescence.
	Using \cite[Proposition II.1.17]{JS} similarly as in the proof of \cite[Theorem III.3.17]{JS}, we obtain that \(\q(a^*_{\tau_n} = 1, \tau_n < \infty) = 1\).
	By the equivalence \(Q \sim P\) it also holds that \(\p(a^*_{\tau_n} = 1, \tau_n < \infty) = 1\). Hence, \(\{a = 1\} \subseteq \{a^* = 1\}\) up to \(\p\)-evanescence.
	For the converse direction we slightly modify the argument.
	Since also \(\{a^* = 1\} \subseteq \{a > 0\}\), there exists a sequence of \(\F\)-predictable times \((\rho_n)_{n \in \mathbb{N}}\) such that \(\{a^* = 1\} = \bigcup_{n \in \mathbb{N}} \of \rho_n\gs\) up to \(\p\)-evanescence.
	Set \(D \triangleq \{\Delta X \not = 0\}\).
	Now \cite[Proposition II.1.17]{JS} yields that \(\q(\rho_n \in D |\mathscr{F}_{\rho_n-}) = a^*_{\rho_n}\) on \(\{\rho_n < \infty\}\) for each \(n \in \mathbb{N}\).
	Hence, we deduce from \cite[Theorem III.3.4]{JS} that
	\(
	\q( \rho_n \not \in D,\rho_n < \infty) 
	= \E_\p[Z_{\rho_n} (1 - a^*_{\rho_n}) \1_{\{\rho_n < \infty\}}] 
	= 0,
	\)
	which implies \(\p(\rho_n \not \in D, \rho_n < \infty) = 0\) since \(Q \sim P\). Using \cite[Proposition II.1.17]{JS} yields that \(\p(a_{\rho_n} = 1, \rho_n < \infty) = 1\) for each \(n \in \mathbb{N}\). This proves that \(\{a^* = 1\} \subseteq \{a = 1\}\) up to \(\p\)-evanescence.
	It follows as in the proof of \cite[Lemma 3.3.1]{liptser1989theory} that  \(\{a^* > 1\}\) is a \(Q\)-evanescence set. Again, since \(Q \sim P\), \(\{a^* > 1\}\) is also a \(P\)-evanescence set.
	Define \(\Lambda' \triangleq \{(\omega, t) \in \Omega \times [0, T]\colon (a_t(\omega) = 1\textup{ and } a^*_t (\omega)\not = 1)\textup{ or }(a_t(\omega) \not = 1\textup{ and } a^*_t (\omega) = 1) \textup{ or } a^*_t(\omega) > 1\}\) and
	\begin{align*}
	U(\omega, t, x) \triangleq \begin{cases}
	1,& \textup{on } \Lambda' \times \mathbb{R},
	\\
	U^*(\omega, t, x),&\textup{otherwise},
	\end{cases}
	\end{align*}
	which is a \(\mathscr{P}(\F)\otimes\mathscr{B}\)-measurable function.
	Recalling \eqref{a leq 1}, we obtain that \(\{U > 0\} = \{a' \leq 1\} = \Omega \times [0,T]\) and that \(\{a = 1\} = \{a' = 1\}\). 
	This finishes the proof.
	\hfill\(\square\)

	\section*{Acknowledgments}
	The author thanks the referees and the associate editor for their time and effort devoted to the evaluation of the manuscript and for their very useful remarks. 
	\appendix\normalsize
	\section{Measurability Lemmata}\label{MLemmata A}
	In this Appendix we collect some measurability results which are used in the proof of Lemma~\ref{lemma Z}.
	We start with an elementary observation.
	\begin{lemma}\label{equi mb}
		A non-negative random variable \(Y\) has an \(\mathscr{H}\)-measurable \(\p\)-version if, and only if, there exists a \(\p\)-null set \(N\in \mathscr{F}\) such that for all \(\omega \in \complement N\) we have \(P(Y = Y(\omega)|\mathscr{H})(\omega)=1\).
	\end{lemma}
	\begin{proof}
		Firstly, we show the implication \(\Longleftarrow\). 
		Thanks to Remark \ref{Remark SA} (i), for all \(A \in \mathscr{F}\) we have \(\E[\1_A \E[Y|\mathscr{H}]] = \E[\1_A\int Y(\omega')P(\dd \omega'|\mathscr{H})] = \E[\1_A Y].\)
		Hence, \(\E[Y|\mathscr{H}]\) is an \(\mathscr{H}\)-measurable \(\p\)-version of \(Y\).
		
		Secondly, assume that \(Y\) has an \(\mathscr{H}\)-measurable \(\p\)-version \(K\). 
		In view of Remark \ref{Remark SA} (ii) and \eqref{varad}, there exists a \(\p\)-null set \(N \in \mathscr{F}\) such that for all \(\omega \in \complement N\) it holds that 
		\(P(Y = Y(\omega)|\mathscr{H})(\omega) = P(K  = K(\omega)|\mathscr{H}) (\omega) = \1_{\{K(\omega)= K(\omega)\}} = 1\). 
	\end{proof}
	Next, we study measurability of integrals w.r.t. random measures. 
	\begin{lemma}\label{ML}
		Assume that \(\nu\) is a \((\G, \p)\)-compensator of a random measure of jumps with an \(\mathscr{H}\)-measurable \(\p\)-version. Let \(U\colon \Omega \times [0, T] \times \mathbb{R} \to \mathbb{R}^+\) be \(\mathscr{P}(\G)\otimes \mathscr{B}\)-measurable such that \(P\)-a.s. \((1 \wedge |x|^2) U \star \nu_T < \infty\) and \(U \cdot \nu\) has an \(\mathscr{H}\)-measurable \(P\)-version.
		\begin{enumerate}
			\item[\textup{(i)}]
			Let \(g \colon \mathbb{R} \times \mathbb{R} \to\mathbb{R}\) be a Borel function such that \(P\)-a.s. \(|g(x, U)| \star \nu_T < \infty\), then \(g(x, U)\star \nu\) has an \(\mathscr{H}\)-measurable \(P\)-version.
			\item[\textup{(ii)}]
			Let \(f \colon \mathbb{R} \times \mathbb{R} \to \mathbb{R}\) be a Borel function such that \(f(0, y) = 0\) for all \(y \in \mathbb{R}\)
			and denote \(a_t \triangleq \nu(\{t\}\times \mathbb{R})\) and \(a'_t \triangleq \int_{\mathbb{R}}U(t, x) \nu (\{t\}\times \dd x)\). Suppose that \(a'_t \leq 1\) for all \(t\in [0, T]\) and \(P\)-a.s. \(\sum_{s \in [0, T]} |f(a_s, a'_s)| < \infty\). The process
			\(
			\sum_{s \in [0, \cdot]} f(a_s, a'_s)
			\)
			has an \(\mathscr{H}\)-measurable \(\p\)-version.
			
		\end{enumerate}
	\end{lemma}
	\begin{proof}
		Denote \(B_n \triangleq \{x \in \mathbb{R} \colon |x| < 1/n\}\)
		and take  \(0 \leq r \leq s \leq T\) and \(G \in \mathscr{B}\).
		There exists a constant \(K\) such that \(\1_{\complement B_n} (x) \leq K(1 \wedge |x|^2)\). 
		Thus, since \(\nu\) and \(U \cdot \nu\) have \(\mathscr{H}\)-measurable \(P\)-versions, the random variables \(\nu((r, s] \times G \cap \complement B_n)\) and \((U \cdot \nu)((r, s] \times G \cap \complement B_n)\) have also \(\mathscr{H}\)-measurable \(\p\)-versions.
		By Remark \ref{Remark SA} (ii) and Lemma \ref{equi mb}, there exists a \(\p\)-null set \(N\in\mathscr{F}\) such that for all \(\omega \in \complement N\) there is a \(P(\cdot|\mathscr{H})(\omega)\)-null set \(N_\omega\in \mathscr{F}\) such that for all \(\omega^* \in \complement N_\omega\) we have 
		\begin{align*}
		(\1_{\complement B_n}U \star \nu_T)(\omega) &+ (\1_{\complement B_n} U\star \nu_T)(\omega^*) + (|g(x, U)| \star \nu_T)(\omega) + (|g(x, U)| \star \nu_T)(\omega^*) \\&+ \sum_{s \in [0, T]} |f(a_s(\omega), a'_s(\omega))| + \sum_{s \in [0, T]} |f(a_s(\omega^*), a'_s(\omega^*))| < \infty
		\end{align*} and 
		\begin{equation}\label{eq: lemma app}
		\begin{split}
		(\1_G \1_{\complement B_n} \star \nu_T)(\omega^*)&= \int_0^T \int_\mathbb{R} \1_G(\omega^*, s, x) \1_{\complement B_n}(x) \nu(\omega, \dd s \times \dd x),\\
		(\1_G \1_{\complement B_n} \star (U\cdot \nu)_T)(\omega^*)&= \int_0^T \int_\mathbb{R} \1_G(\omega^*, s, x) \1_{\complement B_n}(x) U(\omega, s, x) \nu(\omega, \dd s \times \dd x)
		\end{split}
		\end{equation}
		for all \(n \in \mathbb{N}, G = \Omega \times [0, T] \times \mathbb{R}\) and \(G = A\times (r, s] \times (c, d]\) with \(A \in \mathscr{F}, r, s, c, d \in \mathbb{Q}, 0 \leq r \leq s \leq T\) and \(c \leq d\).
		By a monotone class argument, \eqref{eq: lemma app} holds for all \(n \in \mathbb{N}\) and \(G \in \mathscr{F} \otimes \mathscr{B}([0, T]) \otimes \mathscr{B}.\) Letting \(n \to \infty\) and using the monotone convergence theorem yields that 
		\begin{align*}
		(\1_G \star \nu_T)(\omega^*)&= \int_0^T \int_\mathbb{R} \1_G(\omega^*, s, x) \nu(\omega, \dd s \times \dd x),\\
		(\1_G \star (U\cdot \nu)_T)(\omega^*)&= \int_0^T \int_\mathbb{R} \1_G(\omega^*, s, x) U(\omega, s, x) \nu(\omega, \dd s \times \dd x)
		\end{align*}
		for all \(G \in \mathscr{F} \otimes \mathscr{B}([0, T]) \otimes \mathscr{B}.\)
		
		Therefore, for all \(t \in [0, T]\) 
		we have \(a_t(\omega^*) = a_t(\omega)\) and \(a'_t(\omega^*) = a'_t(\omega)\), which implies
		\begin{align*}
		\sum_{s \in [0, t]} f(a_s(\omega^*), a'_s(\omega^*)) = \sum_{s \in [0, t]} f(a_s(\omega), a'_s(\omega)).
		\end{align*}
		By Lemma \ref{equi mb}, this proves the claim of (ii).
		
		Since each non-negative \(\mathscr{F}\otimes \mathscr{B}([0, T])\otimes\mathscr{B}\)-measurable function can be approximated from below by simple non-negative \(\mathscr{F}\otimes\mathscr{B}([0, T])\otimes\mathscr{B}\)-measurable functions, we have 
		\begin{align*}
		(g(x, U)\star \nu_t)(\omega^*) = \int_0^t \int_{\mathbb{R}} g(x, U(\omega^*, s, x)) \nu(\omega, \dd s \times \dd x)
		\end{align*}
		and
		\begin{align*}
		\int_0^T \int_{\mathbb{R}} \1_{G}&(s, x)  U(\omega^*, s, x) \nu(\omega, \dd s \times \dd x) 
		=  \int_0^T \int_\mathbb{R} \1_{G}(s, x) U(\omega, s, x)\nu(\omega, \dd s \times \dd x)
		\end{align*}
		for all \(t \in [0, T]\) and \(G \in \mathscr{B}([0, T])\otimes \mathscr{B}\).    
		Thus,
		\(\nu(\omega, \dd s \times \dd x)\)-a.e. \(U(\omega^*, \cdot, \cdot) = U(\omega, \cdot, \cdot)\). We conclude that for all \(t \in [0, T]\)
		\begin{align*}
		(g(x, U) \star \nu_t)(\omega^*) &= \int_0^t \int_\mathbb{R} g(x, U(\omega^*, s, x)) \nu(\omega, \dd s \times \dd x) = (g(x, U) \star \nu_t)(\omega). 
		\end{align*}
		Now, (i) follows again from Lemma \ref{equi mb}.
	\end{proof}
	The same arguments as in the proof of Lemma \ref{ML} yield the following 
	\begin{lemma}\label{ML2}
		Let \(k \colon \mathbb{R} \to \mathbb{R}^+\) be a Borel function, \(\gamma \colon \Omega \times [0, T]\to \mathbb{R}\) be \(\mathscr{F} \otimes \mathscr{B}([0, T])\)-measurable, \(C \in \mathscr{V}^+\) and assume that \(P\)-a.s. \(|\gamma| \cdot C_T < \infty\). If \(C\) and \(\gamma \cdot C\) have \(\mathscr{H}\)-measurable \(\p\)-versions, then so does \(k(\gamma) \cdot C\).
	\end{lemma}
	
	\section{The Scope of Standing Assumption \ref{SA} 
	}\label{scope}
	
	We give examples for situations where an \((\F, \p)\)-semimartingale is also a \((\G, \p)\)-semi\-martin\-gale. 
	\begin{example}[SIIs]
		If \(\mathscr{H} \triangleq \{\Omega, \emptyset\}\), then \(\F = \G\) and the \((\G, \p)\)- and \((\F, \p)\)-characteristics of \(X\) coincide.
	\end{example}
	\begin{example}[Independent Integrands]\label{Independent Integrands}
		Let \(Y\) be an \(\mathbb{R}^m\)-valued \cadlag process and \(V\) be an \(\mathbb{R}^n\)-valued \cadlag process which are \(\p\)-independent.
		Define 
		\(
		\mathscr{F}^{o}_t \triangleq \sigma(Y_s, V_s, s \in [0, t]), \mathscr{F}^V_t \triangleq \sigma(V_s, s \in [0, t]),
		\F^V \triangleq \left( \mathscr{F}^V_{t+}\right)_{t \in [0, T]},
		\)
		\(\mathbf{F}^Y\) analogously, and \(\mathscr{H} \triangleq \sigma(Y_s, s \in [0, T])\).
		Lemma \ref{lemma countable generated} yields that the \(\sigma\)-fields \(\mathscr{F}^{o}_t\) and \(\mathscr{H}\) are countably generated.
		Assume that \(V\) is an \((\F^V, \p)\)-semi\-martingale whose \((\F^V, \p)\)-characteristics are denoted by \((B^V(h), C^V, \nu^V)\).
		Let \(\mu \colon \mathbb{D}^m \times \mathbb{R}^+ \to \mathbb{R}^d \otimes \mathbb{R}^n\) be such 
		that \(\mu(\cdot, Y) \triangleq\mu(Y) \in L(V, \F^Y, \p)\). 
		Next, we generalizes \cite[Lemma 2.3]{KMK10} to processes without absolutely continuous character\-istics. 
		\begin{lemma}\label{key lemma}
			The process \(X \triangleq \mu(Y) \cdot V\) is an \(\mathbb{R}^d\)-valued \((\G, \p)\)- and \((\F, \p)\)-semimartingale and its \((\G, \p)\)- and \((\F, \p)\)-characteristics \((B(\tilde{h}), C, \nu)\) associated to a truncation function \(\tilde{h}\) are given by
			\begin{equation}\label{bsp cha}
			\begin{split}
			B(\tilde{h})^i &= \sum_{k \leq n} \mu(Y)^{i, k} \cdot B^{V}(h)^k + \left(\tilde{h}^i(\mu(Y) x) - \sum_{k \leq n}\mu(Y)^{i, k} h^k(x)\right) \star \nu^{V},\\
			C^{i,j} & = \sum_{k, l \leq n} \big(\mu(Y)^{i, k} \mu(Y)^{j, l} \big)\cdot C^{V, k,l},\\
			\nu(\dd t \times G) &=  \int_{\mathbb{R}^n} \1_G(\mu(Y) x) \nu^{V}(\dd t \times \dd x),
			\end{split}
			\end{equation}
			for \(i, j \leq d, G \in \mathscr{B}^d, 0 \not \in G\). 
		\end{lemma}
		\begin{proof}
			Thanks to the inclusions \(\F^V \subseteq \F \subseteq \G\), we deduce from \cite[Theorem II.2.42]{JS}, \cite[Satz 15.5]{bauer2002wahrscheinlichkeitstheorie} and the tower rule that \(V\) is an \((\F, \p)\)- and \((\G, \p)\)-semimartingale which \((\F, \p)\)- and \((\G, \p)\)-characteristics are given by \((B^V(h), C^V, \nu^V)\).
			It follows from the inclusions \(\F^Y \subseteq\F \subseteq \G\) and \cite[Theorem III.6.30]{JS} that \(\mu(Y) \in L(V, \F, \p) \cap L(V, \G, \p)\).
			Hence \(\mu(Y) \cdot V\) is an \((\F, \p)\)- and \((\G, \p)\)-semimartingale, whose \((\F, \p)\)- and \((\G, \p)\)-characteristics are given by \eqref{bsp cha},
			cf. \cite[Proposition IX.5.3]{JS}.
		\end{proof}
		Recalling that \(\mu(Y)\) is \(\F^Y\)-predictable, we obtain the following
		\begin{corollary}\label{SA Coro}
			If \((B^V(h), C^V, \nu^V)\) are deterministic, then~\(X = \mu(Y) \cdot V\) is an \((\mathscr{H}, \F, \p)\)-SII and an \((\F, \p)\)-semimartingale whose \((\F, \p)\)- and \((\G, \p)\)-characteristics coincide and are given by \eqref{bsp cha}.
		\end{corollary}
		Corollary \ref{SA Coro} implies that the financial models suggested by \cite{RSSB:RSSB282,Heston01041993,stein1991stock} are exponential \(\mathscr{H}\)-SII models as defined in Section \ref{SPEMM}.
	\end{example}
	\begin{example}[Time-changed L\'evy Models]\label{Time-changed Levy processes}
		We assume that 
		\(V,Y\) and \(\mathscr{H}\) are given as in Example \ref{Independent Integrands} and that \(Y\) is \(\mathbb{R}^+\)-valued.
		Let \(\mu \colon \mathbb{R} \to \mathbb{R}^d\) be a Borel function such that \(\p\)-a.s.
		\(
		|\mu(Y)| \cdot I_T < \infty,
		\)
		where \(\mu(Y)_t \triangleq \mu(Y_t)\).
		Then we set
		\(
		X \triangleq\mu(Y_-) \cdot I + V_{Y_- \cdot I}\) and 
		\(\mathscr{F}^{o}_t \triangleq \sigma(X_s, Y_s, s \in [0, t]).
		\)
		Lemma \ref{lemma countable generated} yields that \(\mathscr{F}^{o}_t\) is countably generated.
		Let 
		\(V\) be an \((\F, \p)\)-L\'evy process with \((\F, \p)\)-L\'evy-Khinchine triplet \((b^{V}(h), c^{V}, F^{V})\).
		The following lemma is a restatement of \cite[Lemma 2.4]{KMK}.
		It shows that the time-changed L\'evy model proposed by \cite{CGMY03} is an exponential \(\mathscr{H}\)-SII model as defined in Section \ref{SPEMM}.
		\begin{lemma}
			\label{remark TC}
			The process \(X\) is an \((\mathscr{H}, \F, \p)\)-SII and an \((\F, \p)\)-semi\-martingale whose \((\G, \p)\)- and \((\F, \p)\)-characteristics coincide and are given by 
			\begin{equation}\label{cha TC}
			\begin{split}
			B(h) &=  \big( \mu(Y_-) + b^{V}(h) Y_-\big)\cdot I,\\ C &= c^{V} Y_-\cdot I,\\ \nu (\dd t \times \dd x) &= Y_{t-} \dd t  F^V(\dd x).
			\end{split}
			\end{equation}
		\end{lemma}
	\end{example}
\bibliographystyle{agsm}
\bibliography{References}

\end{document}